\documentclass[reqno,oneside,12pt]{amsart}
\usepackage{amssymb,enumerate,multirow,bm}
\usepackage{picins}

\usepackage[T1]{fontenc}
\usepackage{mathptmx} 

\usepackage[width=14cm,centering]{geometry}

\newtheorem{theorem}{Theorem}

\newtheorem{cor}[theorem]{Corollary}
\newtheorem{lemma}[theorem]{Lemma}

\theoremstyle{definition}

\theoremstyle{remark}

\newtheorem{example}{Example}

\long\def\red#1{\bgroup
 \color{red}#1\egroup}

\newif\ifcnote\cnotefalse

\cnotetrue

\title[Commuting maps on rank-$k$ matrices]
    {Commuting maps on rank-$k$ matrices}
\author{Willian Franca}
\address{Department of Mathematical Sciences, Kent State
University, Kent, Ohio 44242, USA} \email{wfranca@kent.edu} 
\begin{document}
\maketitle
\begin{abstract}
Let $n\geq2$ be a natural number. Let $M_n(\mathbb{K})$ be the ring of all $n \times n$  matrices over a field $\mathbb{K}$. Fix natural number $k$ satisfying $1<k\leq n$. Under a mild technical assumption over $\mathbb{K}$ we will show that additive maps $G:M_n(\mathbb{K})\to
M_n(\mathbb{K})$ such that $[G(x),x]=0$ for every rank-$k$ matrix $x\in M_n(\mathbb{K})$ are of form $\lambda x + \mu(x)$, where $\lambda\in Z$, $\mu:M_n(\mathbb{K})\to Z$, and $Z$ stands for the center of $M_n(\mathbb{K})$. Furthermore, we shall see an example that there are additive maps such that $[G(x),x]=0$ for all rank-$1$ matrices that are not of the form $\lambda x + \mu(x)$. We will also discuss the $m$-additive case.
\end{abstract}

 Let $n\geq2$ be a natural number, and let $\mathbb{K}$ be a field. An additive map $G:M_n(\mathbb{K})\to M_n(\mathbb{K})$ is called commuting if for each $x\in M_n(\mathbb{K})$ the equality $[G(x),x]=G(x)x-xG(x)=0$ holds. In \cite{WF} we proved that if $G: M_n(\mathbb{K})\to M_n(\mathbb{K}) $ is an additive map and $[G(x),x]=0$ for every invertible $x\in M_n(\mathbb{K})$ then $G$ has the form

\begin{equation}\label{E100}
G(x)=\lambda x + \mu(x),
\end{equation}
where $\lambda\in Z$, $\mu:M_n(\mathbb{K})\to Z$ is an additive map, $Z=\mathbb{K}\cdot I$ is the center of $M_n(\mathbb{K})$, $I$ is the identity matrix, and
$\mathbb{K}$ is any field but $\mathbb{Z}_2$. Thus, writing this result in terms of rank, we could say that: If $G: M_n(\mathbb{K})\to M_n(\mathbb{K})$ is an additive map such that $[G(x),x]=0$ for each rank-$n$ matrix $x\in M_n(\mathbb{K})$ then $G$ has the form (\ref{E100}). So, it gave rise to a natural question: For a fixed natural number $k$ satisfying $1\leq k\leq n-1$  does an additive map $G$ from $M_n(\mathbb{K})$ to itself with the property that $[G(x),x]=0$ for every rank-$k$ matrix $x\in M_n(\mathbb{K})$ have the description (\ref{E100})? Under the assumption that $\mathbb{K}$ is a field that either char $\mathbb{K}=0$ or char $\mathbb{K}>3$ the answer will be yes when $1<k\leq n-1$. In the case that $k=1$ we have a negative answer for $n\geq3.$

Now, let us see how this problem can be formulated in terms of $m$-additive maps. Let $n>m>1$ be natural numbers, and let $\mathbb{K}$ be a field such that char $\mathbb{K}=0$ or char $\mathbb{K}\geq m+1$. We say that $G:M_n(\mathbb{K})^{m}\to M_n(\mathbb{K})$ is $m$-additive if $G$ is additive in each component, that is, $G(x_1,\ldots,x_i+y_i,\ldots,x_m)=G(x_1,\ldots,x_i,\ldots,x_m)+G(x_1,\ldots,y_i,\ldots,x_m)$ for all $x_i,y_i\in M_n(\mathbb{K})$, and $i\in\{1,\ldots,m\}$.  The map $T:M_n(\mathbb{K})\to M_n(\mathbb{K})$ defined by $T(x)=G(x,\ldots,x)$ is known as the trace of $G$, and  such traces are called commuting if for each $x\in M_n(\mathbb{K})$ the equality $[G(x,\ldots,x),x]=0$ holds.

In the present work, we will characterize all $m$-additive maps $G:M_n(\mathbb{K})^m\to M_n(\mathbb{K})$ such that $[G(x,\ldots,x),x]=0$ for every rank-$k$ matrix $x\in M_n(\mathbb{K})$, where $k$ is a fixed natural number satisfying $(m+1)\leq k\leq n$.

We just want to remind that any ($m$-)additive map $G:M_n(\mathbb{K})^{(m)}\to M_n(\mathbb{K})$ is\;\; ($m$-)linear over $\mathbb{Q}$ (respectively $\mathbb{Z}_q$) when char $\mathbb{K}=0$ (respectively char $\mathbb{K}=q$). This fact will be largely used in this paper. Furthermore, it should be mentioned that the set of all nonzero elements of $\mathbb{K}$ will be denoted by $\mathbb{K}^*$.

Let us start with the additive case. First, we need a couple of auxiliary results.
\begin{lemma}\label{E100000}
Let $n\geq3$ be a natural number, and let $M_n(\mathbb{K})$ be the ring of all $n\times n$  matrices over $\mathbb{K}$. Fix a natural number $k$ satisfying $1<k\leq n$, and numbers $i,j\in\{1,\ldots n\}$. Then there exists a matrix $B=B(i,j)\in M_n(\mathbb{K})$ such that $ze_{ij}+tB$ has rank $k$\; for all nonzero $z,t\in \mathbb{K}$.
\end{lemma}
\begin{proof}
Fix $i,j,k$ as described in the statement. For each $v\in\{2,\ldots,k\}$ choose $i_v,j_v\in\{1,\ldots,n\}$ such that  $i\neq i_u, j\neq j_u, i_u\neq i_v$, and $j_u\neq j_v$ when $u\neq v$ ($u,v\in\{2,\ldots,k\})$. Set $B=\sum_{v=2}^{k}e_{i_vj_v}$. Clearly $ze_{ij}+tB$ has rank $k$ for all $z,t\in\mathbb{K}^*$.
\end{proof}

In the same way we can prove the following:

\begin{lemma}\label{E102}
Let $n\geq3$ be a natural number, and let $M_n(\mathbb{K})$ be the ring of all $n\times n$  matrices over $\mathbb{K}$. Fix a natural number $k$ satisfying $1<k\leq n$, and numbers $i,j,k,l\in\{1,\ldots n\},$ where $(i,j)\neq(k,l)$. Then there exists a matrix $B=B(i,j,k,l)\in M_n(\mathbb{K})$ such that $ze_{ij}+we_{kl}+tB$ has rank $k$\; for all nonzero $z,w,t\in\mathbb{K}$.
\end{lemma}

\begin{theorem}\label{T5}
Let $n\geq3$ be a natural number. Let $M_n(\mathbb{K})$ be the ring of all $n
\times n$  matrices over $\mathbb{K}$ with center $Z=\mathbb{K}\cdot I$, where $I$ is the identity matrix. Fix a natural number $k$ satisfying $1<k\leq n-1$. Assume that char $\mathbb{K}=0$ or char $\mathbb{K}>3$. Let $G:M_n(\mathbb{K})\to
M_n(\mathbb{K})$ be an additive map such that
\begin{equation}\label{E31}
[G(x),x]=0 \quad \mbox{for every rank-$k$ matrix $x\in M_n(\mathbb{K})$.}
\end{equation}
Then there exist an element $\lambda \in Z$ and an additive map $\mu: M_n(\mathbb{K}) \to Z$ such that
\begin{equation}\label{E2}
G(x)=\lambda x + \mu(x) \quad \mbox{for each $x\in M_n(\mathbb{K})$.}
\end{equation}
\end{theorem}

\begin{proof}
Let $z,w\in\mathbb{K}$, and $i,j,k,l\in\{1,\ldots,n\}$, where $(i,j)\neq(k,l)$. All the identities that will be obtained during this proof can be easily checked when either $z$ or $w$ are zero. So, we may assume that $z,w$ are nonzero elements of $\mathbb{K}$.  Using Lemma \ref{E100000}, we can find a matrix $B\in M_n(\mathbb{K})$ such that $ze_{ij}+tB$ has rank $k$ for all nonzero $t\in\mathbb{K}$. Then equation (\ref{E31}) tells us that $[G(ze_{ij}+tB),ze_{ij}+tB]=0$ for all $t\in\mathbb{K}^*$. Replacing  $t$ by $1,-1$ and $2$ we arrive at (because char $\mathbb{K}>3$):
\begin{equation}\label{105}
[G(ze_{ij}),ze_{ij}]=0\quad\mbox{for all}\quad z\in\mathbb{K},\quad\mbox{and}\quad i,j\in\{1,\ldots,n\}.
\end{equation}

On the other hand, Lemma \ref{E102} guarantees the existence of  an element $D\in M_n(\mathbb{K})$ such that $ze_{ij}+we_{kl}+tD=c +tD$ has rank $k$. Again, equation (\ref{E31})  provides $[G(c+tD),c+tD]=0$. So, it follows from the same argument that we used in the proof of equation (\ref{105}) that  $[G(c),c]=0$, which becomes
\begin{equation}
[G(ze_{ij}),we_{kl}]+[G(we_{kl}),ze_{ij}]=0,
\end{equation}
since $c=ze_{ij}+we_{kl}$ (note that equation (\ref{105}) has been used).

Hence,
$
[G(ze_{ij}),we_{kl}]+[G(we_{kl}),ze_{ij}]=0,$
for all $i,j,k\;,l\in\{1,\ldots n\}$, and $z,w\in \mathbb{K}$. Thus, $[G(x),x]=0$ for each $x\in M_n(\mathbb{K})$, because $G$ is additive and $M_n(\mathbb{K})$ is additively generated by $ze_{ij}$. The desired result now is obtained from the well-known theorem on commuting maps due to Bre\v sar (see
the original paper \cite{B1}, or the survey paper \cite[Corollary 3.3]{B2}, or the book \cite[Corollary 5.28]{BCM}).

\end{proof}

Now, we will study commuting maps on the set of rank-$1$ matrices. The proof of the next result  can be found in \cite[Theorem 1]{WF} (case $n=2$).

\begin{theorem}\label{T2}
Let  $\mathbb{K}$ be a field, and $n=2$. Let $M_2(\mathbb{K})$  be the ring of all $2 \times 2$  matrices over
$\mathbb{K}$ with center $Z=\mathbb{K}\cdot I$, where $I$ is the identity matrix. Let $G:M_2(\mathbb{K})\to M_2(\mathbb{K})$ be an additive map
such that

\begin{equation}\label{E3}
[G(x),x]=0 \quad \mbox{for every rank-$1$ matrix}\quad x\in M_2(\mathbb{K}).
\end{equation}
Then there exist an element $\lambda \in Z$ and an additive map $\mu: M_2(\mathbb{K}) \to Z$ such that

\begin{equation}\label{E60}
G(x)=\lambda x + \mu(x) \quad \mbox{for each } x\in M_2(\mathbb{K}).
\end{equation}

\end{theorem}

Surprisingly, Theorem \ref{T5} fails when $k=1$ and $n\geq3$.

\begin{example}
Let $\mathbb{K}$  be any field. Fix a natural number $n \geq 3$. Let us define a linear map from $M_n(\mathbb{K})$ to itself in the following way:
$G(ze_{11})=- ze_{n2}$, $G(ze_{1n})= ze_{12}$, $G(ze_{21})= ze_{n1}$, $G(ze_{2n})=\sum_{j=2}^{n} ze_{jj}$, and $G(ze_{ij})=0$ otherwise, where $z\in\mathbb{K}$. The map $G$ is commuting on the set of rank-$1$ matrices. Indeed, let $x$ be a rank-$1$ matrix. Thus, $x$ can be written in the following way:
\begin{equation*}
x=\sum_{j=1}^{n}\lambda_j\sum_{i=1}^{n}x_ie_{ij}\;,
\end{equation*}
where $x_v,\lambda_v\in \mathbb{K}$ for all $v\in\{1,\ldots,n
\}$, and $\lambda_l=1$ for some $l\in\{1,\ldots,n\}$.
Therefore, for the above $x$, we have:
\begin{equation*}
\displaystyle{G(x)=\left(-\lambda_1x_1e_{n2}+\lambda_{n}x_1e_{12}+\lambda_1x_2e_{n1}+\sum_{j=2}^{n}\lambda_nx_2e_{jj}\right)}.
\end{equation*}
By standard computations, we see that:
\begin{equation*}
G(x)x=\lambda_nx_2\left[\sum_{j=1}^{n}\lambda_j\sum_{i=1}^{n}x_ie_{ij}\right]=\\ xG(x).
\end{equation*}
Hence, $[G(x),x]=0$ for all rank-$1$ matrix $x\in M_n(\mathbb{K})$, but $G$ is not of the form (\ref{E100}) as it maps the identity element $I$ to a noncentral element $-e_{n2}$.

\end{example}

Now, we will investigate commuting traces of  $m$-additive maps.\vspace{.2cm}

First of all, observe that we may assume that our $m$-additive map $G:M_n(\mathbb{K})^{m}\to M_n(\mathbb{K})$, whose trace commute on the set of rank-$k$ matrices ($m+1\leq k \leq n$), is symmetric because we can replace $G$ with $\sum_{\sigma \in S_m}G(x_{\sigma(1)},\ldots,x_{\sigma(m)})=G^{'}(x_1,\ldots,x_m)$, and it is clear that $G^{'}$ is symmetric and that $[G^{'}(x,\ldots,x),x]=0$ for each rank-$k$ matrix $x\in M_n(\mathbb{K})$. Notice that $G^{'}(x,\ldots,x)=m!G(x,\ldots,x)$, thus the trace of $G^{'}$ is identically zero if and only if the trace of $G$ is, since either char $\mathbb{K}=0$ or char $\mathbb{K}\geq m+1$.

Secondly, we will state the following result that is a mere generalization of the Lemma \ref{E102}.

\begin{lemma}\label{P1}
Let  $\mathbb{K}$ be a field, and let $m,n$ be natural numbers, where $n\geq m+1,$ and $m>1$. Let $M_n(\mathbb{K})$ be the ring of all $n\times n$  matrices over $\mathbb{K}$. Fix a natural number $k$ satisfying $(m+1)\leq k\leq n$. Let $v$ be a natural number such that $1\leq v\leq m+1$. For each $r\in\{1,\ldots,v\}$, let $i_r,j_r\in\{1,\ldots n\}$, where $(i_u,j_u)\neq(i_v,j_v)$ when $u\neq v$. Then there exists a matrix $B=B((i_r,j_r)|r\in\{1,\ldots,v\})\in M_n(\mathbb{K})$ such that $\sum_{r=1}^{v}z_{r}e_{i_rj_r}+tB$ has rank $k$ for all nonzero $z_1,\ldots,z_v,t\in \mathbb{K}$.

\end{lemma}

The argument needed in the proof of the next Theorem is similar to the one used in \cite[Theorem 2]{WF2}, but for the completeness sake we will provide all the details.
\begin{theorem}\label{T2}
Let  $\mathbb{K}$ be a field, and let $m,n$ be natural numbers, where $n\geq m+1,$ and $m>1$. Let $M_n(\mathbb{K})$  be the ring of all $n \times n$  matrices over
$\mathbb{K}$ with center $Z=\mathbb{K}\cdot I$, where $I$ is the identity matrix. Fix a natural number $k$ satisfying $(m+1)\leq k\leq n$. Let $G:M_n(\mathbb{K})^{m}\to M_n(\mathbb{K})$ be a symmetric $m$-additive map
such that
\begin{equation}\label{E20}
[G(x,\ldots,x),x]=0 \quad \mbox{for every rank-$k$ matrix $x\in M_n(\mathbb{K})$.}
\end{equation}
Assume that char $\mathbb{K}=0$ or char $\mathbb{K}> m+1$. Assume also that $\mathbb{K}$ contains at least $m+4$ elements. Then, there exist $\mu_0\in Z$ and maps $\mu_{i}:M_n(\mathbb{K})\to Z$, $i\in\{1,\ldots,m\}$, such that each $\mu_i$ is the trace of an $i$-additive map and $G(x,\ldots,x)=\mu_0x^m+\mu_1(x)x^{m-1}+\ldots+\mu_{m-1}(x)x+\mu_m(x)$ for each $x\in M_n(\mathbb{K})$.
\begin{proof}
For each $r\in\{1,\ldots,m+1\}$, let $z_{r}\in\mathbb{K}^*$ and let $i_r,j_r\in\{1,\ldots,n\}$, where $(i_u,j_u)\neq(i_v,j_v)$ when $u\neq v$. By Lemma \ref{P1} there exists a matrix $B\in M_n(\mathbb{K})$ such that $\sum_{r=1}^{m+1}z_{r}e_{i_rj_r}+tB=c+tB$ has rank $k$ for all nonzero $t\in\mathbb{K}$. It follows from (\ref{E20}) that $[G(c+tB,\ldots,c+tB),c+tB]=0$ for all $t\in\mathbb{K}^*$. Hence, $[G(c+tB,\ldots,c+tB),c+tB] + [G(c+sB,\ldots,c+sB),c+sB]=0$ for all $s,t\in\mathbb{K}^*$. Letting $0\neq s=-t$, and using the symmetricity of $G$, we get that:
\begin{eqnarray}\label{106}
\displaystyle{\sum_{h=0}^{\zeta}\binom{m}{2h}[G(\underbrace{tB,\ldots,tB}_{2h},c,\ldots,c),c]}+ \displaystyle{\sum_{h=0}^{\epsilon}\binom{m}{2h+1}[G(\underbrace{tB,\ldots,tB}_{2h+1},c,\ldots,c),tB]=0},
\end{eqnarray}
where $\zeta=\frac{m}{2}$, and  $\epsilon=\zeta-1$ when $m$ is even and $\zeta=\epsilon=\frac{m-1}{2}$ when $m$ is odd.
For convenience let us set:
\begin{equation*}
\displaystyle{\alpha(h)=\binom{m}{2h}[G(\underbrace{B,\ldots,B}_{2h},c,\ldots,c),c]},\quad\mbox{where}\quad h\in\{0,\ldots,\zeta\}.
\end{equation*}
\begin{equation*}
\displaystyle{\gamma(h)=\binom{m}{2h+1}[G(\underbrace{B,\ldots,B}_{2h+1},c,\ldots,c),B]},\quad\mbox{where}\quad h\in\{0,\ldots,\epsilon\}.
\end{equation*}

 Observe that for each $t\in\{1,\ldots,\epsilon+2\}$ we have obtained an equation of the form (\ref{106}). It means that we got $(\epsilon+2)$ distinct equations of such form (this is true because char $\mathbb{K}> m+1$, and $|\mathbb{K}|\geq m+4$). Thus, using matrix notation we have the following:
\begin{equation*}
\left(
  \begin{array}{ccccc}
    1 & 1 & 1 & \ldots & 1\\
    1 & 2^{2} & 2^{4} & \ldots & 2^{a}\\
    1 & 3^{2} & 3^{4} & \ldots & 3^{a} \\
    \vdots & \vdots & \vdots & \vdots& \vdots\\
  1 & {(\epsilon+2)}^{2} & {(\epsilon+2)}^{4} & \ldots & (\epsilon+2)^{a}\\
  \end{array}
\right)
\left(
\begin{array}{c}
     \alpha(0)\\
     \alpha(1)+\gamma(0)\\
     \alpha(2)+\gamma(1)\\
     \alpha(3)+\gamma(2)\\
     \vdots\\
     y

\end{array}
\right)=0,
\end{equation*}
where  $a=m,$ (respectively $a=m+1$) and  $y=\alpha(\zeta)+\gamma(\epsilon)$ (respectively $y=\gamma(\epsilon)$)  when $m$ is even (respectively $m$ is odd).

Because the determinant of the Vandermonde matrix formed by the coefficients of the system is not zero, we get that $\alpha(0)=[G(c,\ldots,c),c]=0$, and this implies that:

\begin{equation}\label{E16}
\displaystyle{\left[G(\sum_{r=1}^{m+1}z_{r}e_{i_rj_r},\ldots,\sum_{r=1}^{m+1}z_{r}e_{i_rj_r}),\sum_{r=1}^{m+1}z_{r}e_{i_rj_r}\right]}=0,
\end{equation}
since $c=\sum_{r=1}^{m+1}z_{r}e_{i_rj_r}$.

 Now, fix  $v\in\{1,\ldots,m\}$. Notice that the previous argument works if the element $c=\sum_{r=1}^{m+1}z_{r}e_{i_rj_r}$ is replaced by $c_v=\sum_{r=1}^{v}z_{r}e_{i_rj_r}$, where $z_r\in \mathbb{K}^{*}$, $i_r,j_r\in\{1,\ldots,n\}$ for each $r\in\{1,\ldots,v\}$ with $(i_u,j_u)\neq(i_v,j_v)$ when $u\neq v$. Thus, in the same fashion we can prove that $[G(c_v,\ldots,c_v),c_v]=0$, which yields:
 \begin{equation}\label{E1000}
\displaystyle{\left[G(\sum_{r=1}^{v}z_{r}e_{i_rj_r},\ldots,\sum_{r=1}^{v}z_{r}e_{i_rj_r}),\sum_{r=1}^{v}z_{r}e_{i_rj_r}\right]}=0, \quad \mbox{for all}\quad v\in\{1,\ldots,m\}.
\end{equation}

So, after using that $G$ is $m$-additive and the identity (\ref{E1000}) (for each $v\in\{1,\ldots,m\}$), the equality (\ref{E16}) becomes:
\begin{equation}\label{E101}
\displaystyle{\sum_{\sigma\in S_{m+1}}[G(z_{i_{\sigma(1)}}e_{i_{\sigma(1)}j_{\sigma(1)}},\ldots,z_{i_{\sigma(m)}}e_{i_{\sigma(m)}j_{\sigma(m)}}),z_{i_{\sigma(m+1)}}e_{i_{\sigma(m+1)}j_{\sigma(m+1)}}]=0}.
\end{equation}

Note that the above equation holds trivially when $z_r=0$ for some $r\in\{1,\ldots,m+1\}$. Hence, writing any matrix $x\in M_n(\mathbb{K})$ as the sum of its entries $x_{ij}e_{ij}$ we can conclude from (\ref{E101}) that $[G(x,\ldots,x),x]=0$ for each $x\in M_n(\mathbb{K})$, since $G$ is $m$-additive. The desired result now follows from  \cite[Theorem 3.1]{LW}.

\end{proof}

\end{theorem}

As a consequence of the Theorem \ref{T2} we can easily obtain the following result that was proved in \cite {WF2}.

\begin{cor}
Let  $\mathbb{K}$ be a field, and let $m,n$ be natural numbers, where $n\geq m+1$, and $m>1$. Let $M_n(\mathbb{K})$  be the ring of all $n \times n$  matrices over
$\mathbb{K}$ with center $Z=\mathbb{K}\cdot I$, where $I$ is the identity matrix. Let $G:M_n(\mathbb{K})^{m}\to M_n(\mathbb{K})$ be a symmetric $m$-additive map
such that
\begin{equation}
[G(x,\ldots,x),x]=0, \quad \mbox{for every invertible (singular) $x\in M_n(\mathbb{K})$.}
\end{equation}
Assume that char $\mathbb{K}=0$ or char $\mathbb{K}> m+1$. Assume also that $\mathbb{K}$ contains at least $m+4$ elements. Then, there exist $\mu_0\in Z$ and maps $\mu_{i}:M_n(\mathbb{K}) \to Z$, $i\in\{1,\ldots,m\}$, such that each $\mu_i$ is the trace of an $i$-additive map and $G(x,\ldots,x)=\mu_0x^m+\mu_1(x)x^{m-1}+\mu_{m-1}(x)x+\mu_m(x)$ for each $x\in M_n(\mathbb{K})$.

\end{cor}

\end{document}